\documentclass[12pt, reqno]{amsart}
\usepackage{mathrsfs, amsmath, amsthm, amscd, amsfonts, amssymb}

\newtheorem{theorem}{Theorem}[section]
\newtheorem{lemma}[theorem]{Lemma}
\newtheorem{proposition}[theorem]{Proposition}

\newtheorem{definition}[theorem]{Definition}

\newtheorem{corollary}[theorem]{Corollary}

\newcommand{\bC}{\mathbb{C}}
\newcommand{\bD}{\mathbb{D}}
\newcommand{\bT}{\mathbb{T}}

\newcommand{\bA}{\mathbb{A}}

\newcommand{\bZ}{\mathbb{Z}}

\newcommand{\A}{\mathbb A}

\newcommand{\clb}{\mathcal{B}}

\newcommand{\clh}{\mathcal{H}}

\newcommand{\NI}{\noindent}

\numberwithin{equation}{section}

\begin{document}

\setcounter{page}{1}

\title[Toeplitz operators on the annulus]{Characterizations of Toeplitz operators on the annulus}

\author[Adhikari]{Sayan Adhikari}
\address{Indian Institute of Science, Department of Mathematics, Bangalore, 560012, India}
\email{sayana@iisc.ac.in, sayanadhikari7998@gmail.com}

\author[Das]{Susmita Das}
\address{Indian Institute of Science, Department of Mathematics, Bangalore, 560012, India}
\email{susmitadas@iisc.ac.in, susmita.das.puremath@gmail.com}

\author[Jain]{Shubham Jain}
\address{Indian Statistical Institute, Statistics and Mathematics Unit, 8th Mile, Mysore Road, Bangalore, 560059, India}
\email{shubhamjain\_pd@isibang.ac.in, shubhamjain4343@gmail.com}

\author[Sarkar]{Jaydeb Sarkar}
\address{Indian Statistical Institute, Statistics and Mathematics Unit, 8th Mile, Mysore Road, Bangalore, 560059, India}
\email{jay@isibang.ac.in, jaydeb@gmail.com}




\begin{abstract}
We establish algebraic characterizations of Toeplitz operators on the Hardy space on the annulus. As an application, we characterize compact perturbations of Toeplitz operators and prove a zero-product result.
\end{abstract}

\maketitle

\tableofcontents

\section{Introduction}

Let $r \in (0,1)$ be a fixed number, and let $\bA$ denote the annulus centered at the origin with inner radius $r$ and outer radius $1$:
\[
\bA = \{z \in \mathbb C: r < |z| < 1 \}.
\]
Define the Hardy space on $\A$ as
\begin{align*}
H^2(\A)  =\Bigg\{f: \bA \rightarrow \bC \text{ holomorphic}: \sup_{r<t<1}\int_{0}^{2\pi}|f(te^{i\theta})|^2d m(\theta) < \infty \Bigg\},
\end{align*}
where $m$ denotes the Lebesgue measure on $\mathbb{R}$. This space was studied in detail by Sarason \cite{S1965}. It has since been studied from various perspectives. See, for example, \cite{A1974, AD, JA, MC, MS}, and the recent work \cite{DN2025, MP}, as well as the references therein.

The purpose of this paper is to present characterizations of Toeplitz operators on $H^2(\A)$. We then use this characterization to describe compact perturbations of Toeplitz operators on $H^2(\A)$, which we call asymptotic Toeplitz operators, following Barr\'{i}a and Halmos \cite{BH1982} and Feintuch \cite{F1988}. We point out that Abrahamse in \cite{A1974} extended several properties of classical Toeplitz operators on the Hardy space of the open unit disc to Hardy spaces on the annulus (in fact, on multiply connected domains). To carry out our discussion, we introduce a well-known representation of $H^2(\A)$ based on the boundary values of its functions.

Note that the boundary of the annulus consists of two disjoint circles:
\[
\partial \A =\bT \sqcup r \bT,
\]
where $\bT = \{z\in\bC: |z|=1\}$ and $r \bT := \{r z : z \in \bT\}$. Define $L^2(\partial \A)$ as the Hilbert space of all square-integrable functions on $\partial \A$:
\[
L^2(\partial \bA ) = \Big\{f:\partial \bA \longrightarrow \bC \text{ measurable}: \|f\|_{\partial \bA} := \Big(\int_{\partial\bA}
|f(z)|^2d\sigma(z)\Big)^{\frac{1}{2}} <\infty\Big\},
\]
where $\sigma$ denotes the sum of the normalized arc-length measures on $\mathbb T$ and $r\mathbb T$, so that $\sigma(\mathbb T)=\sigma(r\mathbb T)=1$ (see Section \ref{sec: TO} for more details). Denote by $H^2(\partial \A)$ the closure in $L^2(\partial\A)$ of the set of rational functions on $\mathbb C$ with no poles in $\overline{\A}$. Because functions in $H^2(\bA)$ have non-tangential limits at almost every point of $\partial \A$, each function in $H^2(\A)$ corresponds to a function in $H^2(\partial \A)$. This correspondence is an isometric isomorphism and is onto. We therefore identify these spaces and use the notations interchangeably (see \cite[Theorem 1]{S1965}).

Finally, let $L^\infty(\partial \A)$ denote the space of all essentially bounded functions on $\partial \A$. This space forms a Banach algebra under the essential supremum norm. For $\varphi \in L^\infty(\partial \A)$, the \textit{Toeplitz operator} $T_{\varphi}$ on $H^2(\partial \bA)$ with symbol $\varphi$ is defined by
\[
T_\varphi h = P_{H^2(\partial \A)} (\varphi h),
\]
for all $h \in H^2(\partial \A)$, where $P_{H^2(\partial \A)}$ denotes the orthogonal projection from $L^2(\partial \A)$ onto $H^2(\partial \A)$. The simplest nontrivial example of a Toeplitz operator corresponds to the holomorphic symbol $\varphi(z) = z$, $z \in \A$. We denote it by $M_z$ and call it the shift operator on $H^2(\A)$. In this case,
\[
M_z h = T_z h = zh,
\]
for all $h \in H^2(\partial \A)$. We now present our algebraic characterizations of Toeplitz operators. For a Hilbert space $\clh$, let $\clb(\clh)$ denote the space of bounded linear operators on $\clh$.

\begin{theorem}\label{BHthm}
Let $T \in \clb(H^2(\A))$. The following are equivalent:
\begin{enumerate}
\item $T$ is a Toeplitz operator.
\item $-M_z^{*2}TM_z^{2} + (1+r^2)M_z^{*}TM_z = r^2T$.
\item There exist unique operators $X, Y \in \clb(H^2(\A))$ such that $X+Y=T$ and
\[
M_z^* X M_z = X \text{ and } M_z^* Y M_z = r^2 Y.
\]
\end{enumerate}
\end{theorem}

The above result is in the spirit of Brown and Halmos’s algebraic characterization of Toeplitz operators on the Hardy space of the unit disc (see \cite[Theorem 6]{BH1964}, also see \eqref{eqn-BH TO}). It is convenient and instructive to introduce notation for the shift operator on $H^2(\bA)$ and its scaled inverse. We denote them by $S_1$ and $S_2$, respectively:
\begin{equation}\label{eqn-S1 Sr}
S_1 = M_z \text{ and } S_2 = r M_z^{-1}.
\end{equation}
To clarify the roles of $S_1$ and $S_2$, note that the condition in (3) of the above result, $M_z^* Y M_z = r^2 Y$, is equivalent to (see \eqref{eqn MYM = r2Y})
\[
S_2^* Y S_2 = Y.
\]
Also, evidently, $M_z^* X M_z = X$ is the same as stating that $S_1^* X S_1 = X$.

As a by-product of the above result, we characterize compact perturbations of Toeplitz operators (see Theorem \ref{thm: ATO}): An operator $T \in \clb(H^2(\A))$ is a compact perturbation of a Toeplitz operator if and only if there exist $J_1, J_2 \in \clb(H^2(\A))$ such that
\[
S_i^{*m}T S_i^m \longrightarrow J_i \mbox{ as } m \rightarrow \infty,
\]
in norm for all $i = 1, 2$, and
\[
S_p^{*m} \bigl(T - J_1 - J_2 \bigr) S_q^m \longrightarrow 0 \mbox{ as } m \rightarrow \infty,
\]
in norm for all $p, q=1,2$.

The corresponding result for the Hardy space on the disc is due to Feintuch \cite[Theorem 4.1]{F1988}. The case of the Hardy space on the polydisc was established in \cite{MSS2018}; see also the remark at the end of Section \ref{sec: CO}.

In Theorem \ref{Zero Pdct. Thm}, we prove a zero-product result that fully extends the earlier result in \cite[Theorem 1.3]{DN2025}: Suppose $\varphi, \psi \in L^\infty(\partial \mathbb A)$ and $T_{\varphi}T_{\psi}=0$. Then
\[
\varphi=0 \text{ or } \psi=0  \text{ a.e. on } \partial \A.
\]

The remainder of the paper is organized as follows. Section \ref{sec: TO} establishes the main result, namely, characterizations of Toeplitz operators. Section \ref{sec: CO} characterizes compact operators in terms of shift operators. In Section \ref{sec: ATO}, we apply these characterizations to describe compact perturbations of Toeplitz operators. The final section, Section \ref{sec: zero}, presents a zero-product result on the annulus.

\section{Toeplitz operators}\label{sec: TO}

We begin with some basic background on $H^2(\A)$. For more details, we refer the reader to Sarason’s seminal work \cite{S1965}.

First, we recall that $L^2(\partial \A)$ is the space of all square-integrable, scalar-valued functions on $\partial \A$ with respect to the measure $\sigma$, defined for every measurable set $E\subseteq \partial \A$ by
\[
\sigma (E)=\frac{1}{2\pi}\Big(\big(m\{a\in[0, 2\pi): e^{ia}\in E\} \big )+\big(m\{b\in[0, 2\pi): re^{ib}\in E\} \big )\Big).
\]
In view of $\partial \A = \mathbb T \sqcup r\mathbb T$, the inner product on $L^2(\partial \A)$ is given by:
\[
\langle f,g\rangle_{L^2(\partial \A)} = \int_{\partial\A} f(z)\overline{g(z)}\,d\sigma(z)\notag = \frac{1}{2\pi}\int_{0}^{2\pi}f(e^{it})\overline{g(e^{it})} d m(t) + \frac{1}{2\pi}\int_{0}^{2\pi}f(r e^{it})\overline{g(r e^{it})} d m(t),
\]
for all $f, g\in L^2(\partial \A)$. In particular,
\[
\|f\|_{L^2(\partial \A)}^2 = \int_{\partial\A} |f(z)|^2d\sigma(z)=\frac{1}{2\pi}\int_{0}^{2\pi}|f(e^{it})|^2 dm(t) + \frac{1}{2\pi}\int_{0}^{2\pi}|f(re^{it})|^2 d m(t),
\]
for all $f \in L^2(\partial \A)$. In view of this, for each $h\in H^2(\A)$, we have
\[
\|h\|_{H^2(\A)}^2 = \|h\|_1^2+\|h\|_r^2,
\]
where
\[
\|h\|_1^2 := \frac{1}{2\pi} \int_0^{2\pi}|h(e^{it})|^2 d m(t), \mbox{ and } \|h\|_r^2
:=
\frac{1}{2\pi} \int_0^{2\pi}|h(re^{it})|^2 d m(t).
\]

As noted in the introduction, we identify $H^2(\bA)$ with $H^2(\partial \A)$, the closure in $L^2(\partial\A)$ of the rational functions on $\bC$ with no poles in $\overline{\A}$. This identification is an onto isometry. We will use the two notations interchangeably. We also use the following orthonormal basis for $H^2(\A)$, given by $\{e_n(z)\}_{n\in\mathbb{Z}}$ (cf. \cite{H2020}):
\begin{equation}\label{BasisofhARDY}
e_n(z)=\frac{1}{\sqrt{1+r^{2n}}}z^n \quad (n \in \bZ).
\end{equation}

We now turn to our main goal and prove the first result. Before doing so, we introduce a few more pieces of notation. As noted earlier, $M_z$ denotes the multiplication operator by the coordinate function $z$ on $H^2(\A)$. Recall that (see \cite[page 47]{S1965})
\[
\text{spectrum of } M_z = \overline{\A},
\]
so, in particular, $M_z$ is invertible on $H^2(\A)$:
\[
M_z^{-1} \in \clb(H^2(\A)).
\]
Recall from \eqref{eqn-S1 Sr} that $S_1$ and $S_2$ act on $H^2(\bA)$, where
\[
S_1 = M_z \text{ and } S_2 = r M_z^{-1}.
\]
Clearly,
\[
S_1 S_2 = r I_{H^2(\partial \A)}.
\]

Alongside the Hardy space on the annulus, we also consider the Hardy space on the unit disc $\bD: = \{z \in \bC: |z| < 1\}$, denoted by $H^2(\bD)$. Using boundary values, we identify $H^2(\bD)$ with the closed subspace of $L^2(\bT)$ consisting of functions whose negative Fourier coefficients vanish. For each symbol $\varphi \in L^\infty(\bT)$, the \textit{Toeplitz operator} $T_{\varphi}^{H^2(\bD)}$ on $H^2(\bD)$ is defined by
\[
T_{\varphi}^{H^2(\bD)} h = P_{H^2(\bD)}(\varphi h) \qquad (h \in H^2(\bD)),
\]
where $P_{H^2(\bD)}$ denotes the orthogonal projection of $L^2(\bT)$ onto $H^2(\bD)$. The Toeplitz operator on $H^2(\bD)$ with symbol $\varphi(z) = z$, $z \in \bD$, is the well-known \textit{shift} operator, which we denote by $S$. Therefore,
\[
S h = T_{z}^{H^2(\bD)} h = z h,
\]
for all $h \in H^2(\bD)$. We recall the classical result of Brown and Halmos \cite[Theorem 6]{BH1964}: An operator $T \in \clb(H^2(\bD))$ is a Toeplitz operator if and only if
\begin{equation}\label{eqn-BH TO}
S^*TS=T.
\end{equation}

We now connect the Hardy spaces on the disc and the annulus, together with the operators $S$, $S_1$, and $S_2$, through the following elementary lemma. This lemma is central to our characterization of Toeplitz operators on the annulus and may also be of independent interest.

\begin{lemma}\label{embeddings}
Define $G_1, G_2 : H^2(\bD) \to H^2(\A)$ by
\[
(G_1 h)(z):= h(z),
\]
and
\[
(G_2 h)(z):= h\left(\frac{r}{z}\right),
\]
for all $h \in H^2(\bD)$ and $z \in \bA$. Then the following hold:
\begin{enumerate}
\item $G_1$ and $G_2$ are left-invertible operators with $\|G_1\| = \|G_2\| = \sqrt{2}$.
\item $G_1 S = S_1 G_1$.
\item $G_2 S = S_2 G_2$.
\end{enumerate}
\end{lemma}
\begin{proof}
For any $h = \sum_{n=0}^\infty c_n z^n \in H^2(\bD)$, we have
\[
(G_1 h)(z) = \sum_{n=0}^\infty c_n z^n,
\]
and
\[
(G_2 h)(z) = \sum_{n=0}^\infty c_n r^n z^{-n},
\]
for all $z \in \A$. Using the orthogonality of $\{z^n\}_{n \in \bZ}$ in $H^2(\bA)$ and the fact that (see \eqref{BasisofhARDY})
\[
\|r^n z^{-n}\|_{H^2(\bA)}^2 = r^{2n}(1 + r^{-2n}) = 1 + r^{2n},
\]
for all $n \in \bZ_+$, we obtain
\begin{align*}
\|G_1 h\|_{H^2(\bA)}^2 = \|G_2 h\|_{H^2(\A)}^2 = \sum_{n=0}^\infty (1 + r^{2n}) |c_n|^2 \le 2 \sum_{n=0}^\infty |c_n|^2 = 2 \|h\|_{H^2(\bD)}^2,
\end{align*}
with equality when $h$ is a nonzero constant function. Thus, $\|G_1\| = \|G_2\| = \sqrt{2}$. This estimate also shows that $G_1$ and $G_2$ are bounded below and therefore left-invertible. Moreover, for any $h \in H^2(\bD)$ and $z \in \A$, we have
\[
(S_1 G_1 h)(z) = z h(z) = (S h)(z) = (G_1 S h)(z),
\]
which yields $G_1 S = S_1 G_1$. Similarly,
\[
(S_2 G_2 h)(z) = \frac{r}{z} h\left(\frac{r}{z}\right) = (S h)\left(\frac{r}{z}\right) = (G_2 S h)(z),
\]
so $G_2 S = S_2 G_2$.
\end{proof}

We also need to analyze $L^\infty(\partial \A)$, the space of essentially bounded measurable functions on $\partial \A$ with respect to the measure $\sigma$.

As noted in the introduction,
$\partial\mathbb A=\mathbb T\sqcup r\mathbb T$.
For $b\in\{1,r\}$, let $m_b$ denote the normalized arc-length measure on $b\mathbb T$, defined by
\begin{align*}
m_b(E)=\frac{1}{2\pi} m\big(\{t\in[0,2\pi):be^{it}\in E\}\big),
\end{align*}
for every measurable set $E\subseteq b\mathbb T$. Then, for every measurable set $E\subseteq\partial\mathbb A$, we have
\begin{align*}
\sigma(E)=m_1(E\cap\mathbb T) +m_r(E\cap r\mathbb T).
\end{align*} In particular,
\begin{align*}
\sigma(E)=0 \iff m_1(E\cap\mathbb T)=0\quad\text{and}\quad m_r(E\cap r\mathbb T)=0.
\end{align*}
In view of this, $L^\infty(\partial \A)$ admits a natural decomposition. Consider the Banach space $L^\infty(\bT, m_1) \oplus L^\infty(r \bT, m_r)$, which we abbreviate as
\[
L^\infty(\bT) \oplus L^\infty(r \bT),
\]
equipped with the norm
\[
\|(g, h)\|_\infty = \max\{\|g\|_{L^\infty(\bT)}, \|h\|_{L^\infty(r \bT)}\}.
\]
It is routine to verify that the map
\[
\Psi(\varphi) = \left(\varphi|_{\mathbb{T}}, \varphi|_{r \bT}\right) \qquad (\varphi \in L^\infty(\partial \A)),
\]
defines an isometric isomorphism from $L^\infty(\partial \bA)$ onto $L^\infty(\bT) \oplus L^\infty(r \bT)$. For instance, to see that $\Psi$ is surjective, pick $(g, h) \in L^\infty(\mathbb{T}) \oplus L^\infty(r \bT)$ and define $\varphi : \partial \A \to \bC$ by
\[
\varphi = g \chi_{\bT} + h \chi_{r \bT}.
\]
It follows that $\varphi \in L^\infty(\partial \bA)$ and $\Psi(\varphi) = (g, h)$.

In view of the above, we will often identify elements of $L^\infty(\partial \bA)$ with elements of $L^\infty(\bT) \oplus L^\infty(r \bT)$ via the isometric isomorphism $\Psi$, without explicitly mentioning the identification.

The final ingredient in our characterization is a formula for the matrix coefficients of Toeplitz operators, expressed using the Hardy space inner product. Let $\varphi\in L^2(\partial \A)$. For each $n\in \mathbb{Z}$, define the pair of  $n$-th Fourier coefficients by
\[
\widehat{\varphi_\bT}(n)=\frac{1}{2\pi}\int_{0}^{2\pi}\varphi(e^{it})e^{-int}dm(t),
\]
and
\[
\widehat{\varphi_{r \bT}}(n) = \frac{1}{2\pi}\int_{0}^{2\pi}\varphi(re^{it})e^{-int}d m(t).
\]
If $\varphi \in L^\infty(\partial \A)$, then for every $j,k\in\mathbb Z$, we have (cf. \cite[page 988]{DN2025})
\begin{equation}\label{toeplitz-coefficientss}
\langle T_\varphi z^k,z^j\rangle = \widehat{\varphi}_{\bT}(j-k) + r^{j+k} \widehat{\varphi}_{r \bT}(j-k).
\end{equation}

We recall that $S_1 = M_z$ and $S_2 = r M_z^{-1}$, so $S_1 S_2 = r I_{H^2(\A)}$. Therefore, for $Y \in \clb(H^2(\A))$, the condition
\[
S_2^* Y S_2 = Y,
\]
is equivalent to
\begin{equation}\label{eqn MYM = r2Y}
M_z^* Y M_z = r^2 Y.
\end{equation}
We are now ready to prove the central result of this paper.

\begin{theorem}\label{thm: TO}
Let $T \in \clb(H^2(\A))$. The following are equivalent:
\begin{enumerate}
\item $T$ is a Toeplitz operator.
\item $-S_1^{*2} T S_1^{2} + (1+r^2) S_1^{*} T S_1 = r^2T$.
\item There exists unique operators $X, Y \in \clb(H^2(\A))$ such that $X+Y=T$ and
\[
S_1^* X S_1 = X \text{ and } S_2^* Y S_2 = Y.
\]
\end{enumerate}
\end{theorem}
\begin{proof}
Suppose first that $T=T_\varphi$ for some $\varphi\in L^\infty(\partial\A)$. For each $f,g\in H^2(\A)$, we have
\begin{align*}
&\left\langle\left(-M_z^{*2}T_\varphi M_z^2+(1+r^2)M_z^*T_\varphi M_z-r^2T_\varphi\right)f,g\right\rangle
\\
&\qquad \qquad = \int_{\partial\A}\varphi(z)\left[-|z|^4+(1+r^2)|z|^2-r^2\right]f(z)\overline{g(z)}\,d\sigma(z).
\end{align*}
Since $|z|=1$ for all $z \in \bT$ and $|z|=r$ for all $z \in r\bT$, we have
\[
-|z|^4+(1+r^2)|z|^2-r^2 = 0,
\]
on both $\bT$ and $r\bT$. Recalling that $S_1=M_z$, we obtain
\begin{align*}
-S_1^{*2}T_\varphi S_1^2+(1+r^2) S_1^*T_\varphi S_1 = r^2T_\varphi.
\end{align*}
This proves $(1)\implies(2)$. To prove $(2)\implies (3)$, we first note a general fact: if we define
\begin{equation}\label{eqn: X & Y}
X = \frac{1}{1 - r^2} (S_1^* T S_1 - r^2 T) \text{ and } Y = \frac{1}{1 - r^2} (T - S_1^* T S_1),
\end{equation}
then
\begin{equation}\label{eqn: X+Y=T}
X + Y = T.
\end{equation}
In addition, note that
\[
S_1^* X S_1 = M_z^* X M_z = \frac{1}{1-r^2}(M_z^{*2}TM_z^2-r^2M_z^*TM_z).
\]
Now we turn to our assumption that $- M_z^{*2} T M_z^{2} + (1+r^2) M_z^* T M_z = r^2T$. Since $M_z^{*2} T M_z^{2} = (1+r^2) M_z^* T M_z - r^2T$, it follows that
\[
M_z^* X M_z = \frac{1}{1-r^2}((1+r^2)M_z^{*}TM_z-r^2T-r^2M_z^*TM_z)=X.
\]
Similarly,
\[
M_z^*YM_z = \frac{1}{1-r^2}(M_z^{*}TM_z-M_z^{*2}TM_z^2) = \frac{1}{1-r^2} (-(1 + r^2) M_z^* T M_z + r^2T + M_z^*TM_z).
\]
Thus, $M_z^*YM_z = r^2Y$, and hence, by \eqref{eqn MYM = r2Y}, we have $S_2^*YS_2=Y$. We postpone the proof that the representation of $T$ in \eqref{eqn: X+Y=T} is unique until the end of this section.

\NI Now we prove that $(3)\implies (1)$. In this part of the proof, we repeatedly use the notation and conclusions of Lemma \ref{embeddings}. We know that
\[
S_1^* X S_1 = X \qquad \text{and} \qquad S_2^* Y S_2 = Y.
\]
Define $A_1, A_2 \in \clb(H^2(\bD))$ by
\begin{align*}
A_1 := G_1^* X G_1, \qquad A_2 := G_2^* Y G_2.
\end{align*}
We claim that $A_1$ and $A_2$ are Toeplitz operators on $H^2(\bD)$. Indeed, since $S_1^* X S_1 = X$ and $S^* G_1^* = G_1^* S_1^*$ (see Lemma \ref{embeddings}), it follows that
\[
S^* A_1 S = S^* (G_1^* X G_1) S = (G_1^* S_1^*) X (S_1 G_1) = G_1^* (S_1^* X S_1) G_1 = G_1^* X G_1 = A_1,
\]
proving that $A_1$ is a Toeplitz operator. Similarly,
\[
S^*A_2 S = S^*(G_2^*YG_2)S=(G_2^*S_2^*)Y(S_2 G_2)=G_2^*(S_2^*YS_2)G_2=G_2^*YG_2=A_2,
\]
showing that $A_2$ is a Toeplitz operator (see \eqref{eqn-BH TO}). By Brown and Halmos \eqref{eqn-BH TO}, there exist unique symbols $\varphi_1, \widetilde\varphi \in L^\infty(\bT)$ such that
\[
A_1 = T_{\varphi_1}^{H^2(\bD)} \qquad \text{and} \qquad A_2 = T_{\widetilde\varphi}^{H^2(\bD)}.
\]
Define $\varphi_r \in L^\infty(r\bT)$ by
\[
\varphi_r(r e^{i\theta}) := \widetilde\varphi(e^{-i\theta}).
\]
Via the identification provided by the onto isometry $\Psi$, outlined preceding the statement of this theorem, we know that $(\varphi_1, 0)$ and $(0, \varphi_r)$ belong to $L^\infty(\partial \A)$. We claim that
\[
X = T_{(\varphi_1, 0)} \qquad \text{ and } \qquad Y = T_{(0, \varphi_r)},
\]
on $H^2(\A)$. To show this, it suffices to show that matrix coefficients of $X$ and $Y$ agree with those of the Toeplitz operators $T_{(\varphi_1, 0)}$ and $ T_{(0, \varphi_r)}$, respectively (see \eqref{toeplitz-coefficientss}). To this end, fix $k, n \in \bZ$. Choose an integer $m \ge 1$ large enough that
\[
k + m, n + m \ge 0.
\]
Then
\[
z^{k+m}, z^{n+m} \in H^2(\bD),
\]
and by definition of $G_1$,
\[
G_1(z^{k+m}) = z^{k+m} = S_1^m z^k,
\]
and
\[
G_1(z^{n+m}) = z^{n+m} = S_1^m z^n.
\]
Since $S_1^* X S_1 = X$, it follows by iteration that $S_1^{*m} X S_1^m = X$. Therefore,
\begin{align*}
\langle X z^k, z^n\rangle_{H^2(\bA)} & = \langle S_1^{*m} X S_1^m z^k, z^n\rangle_{H^2(\bA)}
\\
& = \langle X G_1 z^{k+m}, G_1 z^{n+m} \rangle_{H^2(\A)}
\\
& = \langle G_1^* X G_1 z^{k+m}, z^{n+m}\rangle_{H^2(\bD)}
\\
& = \langle T_{\varphi_1}^{H^2(\bD)} z^{k+m}, z^{n+m}\rangle_{H^2(\bD)}
\\
& = \widehat{\varphi_1}\big((n + m) - (k + m)\big)
\\
& = \widehat{\varphi_1}(n - k)
\\
& = \langle T_{(\varphi_1, 0)} z^k, z^n\rangle_{H^2(\bA)},
\end{align*}
where the last equality uses \eqref{toeplitz-coefficientss}. In the above, we also used the well-known fact that the matrix entries of the Toeplitz operator $T_{\psi}^{H^2(\bD)}$ on $H^2(\bD)$, for symbol $\psi \in L^\infty(\bT)$, with respect to the orthonormal basis $\{z^k\}_{k \in \bZ_+}$, are given by
\[
\widehat{\psi}(p - q) = \langle T_{\psi}^{H^2(\bD)} z^q, z^p\rangle_{H^2(\bD)},
\]
for all $p,q \in \bZ_+$. Since $\{z^p : p \in \bZ\}$ is an orthogonal basis for $H^2(\A)$, it follows that
\[
X = T_{(\varphi_1, 0)}.
\]
To prove that $Y = T_{(0, \varphi_r)}$, we again choose an integer $m$ large enough so that
\[
m - k, m - n \ge 0.
\]
Since $S_2 = r M_z^{-1}$, we have
\[
S_2^m z^k = r^m z^{k-m} = r^k \Big(\frac{r}{z}\Big)^{m-k} = r^k G_2(z^{m-k}),
\]
and similarly,
\[
S_2^m z^n = r^n G_2(z^{m-n}).
\]
Proceeding as before, $S_2^* Y S_2 = Y$ implies $S_2^{*m} Y S_2^m = Y$, and hence
\begin{align*}
\langle Y z^k,z^n\rangle_{H^2(\A)} & = \langle S_2^{*m} Y S_2^m z^k,z^n\rangle_{H^2(\A)} \notag
\\
& = r^{n+k} \langle Y G_2 z^{m-k}, G_2 z^{m-n} \rangle_{H^2(\A)}
\\
& = r^{n+k} \langle G_2^* Y G_2 z^{m-k}, z^{m-n}\rangle_{H^2(\bD)}
\\
&= r^{n+k} \langle T_{\widetilde{\varphi}}^{H^2(\bD)} z^{m-k}, z^{m-n}\rangle_{H^2(\bD)}
\\
&= r^{n+k} \widehat{\widetilde{\varphi}}\big((m - n) - (m - k)\big)
\\
&= r^{n+k} \widehat{\widetilde{\varphi}}(k - n)
\\
& = r^{n+k} \widehat{\varphi_r}(n - k)
\\
& = \langle T_{(0, \varphi_r)} z^k, z^n\rangle_{H^2(\bA)}.
\end{align*}
Therefore,
\[
Y = T_{(0, \varphi_r)}.
\]
We know that $X + Y = T$. Since
\[
\varphi:=(\varphi_1, \varphi_r) = (\varphi_1, 0) + (0, \varphi_r) \in L^\infty(\partial \A),
\]
and
\[
X + Y = T_{(\varphi_1, 0)} + T_{(0, \varphi_r)} = T_{\varphi},
\]
we conclude that $T = T_{\varphi}$ is a Toeplitz operator, completing the proof that $\textup{(3)} \implies \textup{(1)}$. This concludes the proof of the theorem.
\end{proof}

We now prove that the representation of $T$ in part $(3)$ of Theorem \ref{thm: TO} is unique. More specifically, we want to prove that the representation $T = X + Y$ in \eqref{eqn: X+Y=T} is unique, where $X$ and $Y$ are given by (see \eqref{eqn: X & Y})
\[
X = \frac{1}{1 - r^2} (S_1^* T S_1 - r^2 T) \text{ and } Y = \frac{1}{1 - r^2} (T - S_1^* T S_1).
\]
Suppose
\[
T = \widetilde X + \widetilde Y,
\]
where $\widetilde X, \widetilde Y \in \clb(H^2(\A))$ satisfy
\[
S_1^* \widetilde X S_1 = \widetilde X \text{ and } S_2^* \widetilde Y S_2 = \widetilde Y.
\]
Now $T = \widetilde X + \widetilde Y$ implies $S_1^* T S_1 = \widetilde X + S_1^* \widetilde Y S_1$. Since $S_1^* \widetilde Y S_1 = r^2 \widetilde{Y}$, it follows that
\[
S_1^* T S_1 - r^2 T = \widetilde X + S_1^* \widetilde Y S_1 - r^2 \widetilde{X} - r^2 \widetilde{Y} = (1- r^2) \widetilde{X},
\]
which yields $\widetilde X = \frac{1}{1 - r^2} (S_1^* T S_1 - r^2 T)$. Since $\widetilde Y = T - \widetilde X$, we finally conclude that
\[
\widetilde X = \frac{1}{1 - r^2} (S_1^* T S_1 - r^2 T) \text{ and } \widetilde Y = \frac{1}{1 - r^2} (T - S_1^* T S_1).
\]
In other words, $\tilde{X} = X$ and $\tilde{Y} = Y$, where $X$ and $Y$ are given by \eqref{eqn: X & Y}. This completes the uniqueness part of the proof of $(2) \implies (3)$, which was previously left to be established.

\section{Compact operators}\label{sec: CO}

Compact operators constitute one of the most important classes of operators in infinite-dimensional analysis. They are often regarded as “small” operators and provide an important link between finite- and infinite-dimensional theory. There are several characterizations of compactness; here, we focus on one involving the shift operator on $H^2(\A)$. This problem for $H^2(\bD)$ was studied extensively by Barr\'ia and Halmos \cite{BH1982} and Feintuch \cite{F1988}. Here, we continue this line of work on $H^2(\A)$.

We continue with the notation as in \eqref{eqn-S1 Sr}: $S_1=M_z$ and $S_2 = rM_z^{-1}$. Also, recall from the beginning of the previous section that $\|h\|^2_{H^2(\A)} = \|h\|_1^2 + \|h\|_r^2$ for all $h \in H^2(\A)$. Therefore,
\begin{align*}
\|S_1 h\|^2 = \|M_z h\|^2 = \|h\|_1^2 + r^2 \|h\|_r^2 \leq \|h\|^2,
\end{align*}
and
\begin{align*}
\|S_2h\|^2 = \|rM_z^{-1}h\|^2 = \|rM_{z^{-1}}h\|^2 = r^2\|h\|_{1}^2 + \|h\|_{r}^2 \leq\|h\|^2,
\end{align*}
for all $h \in H^2(\A)$ (note that $M_z^{-1} = M_{z^{-1}}$ on $H^2(\A)$). In particular, $S_1$ and $S_2$ are contractions. Since $e_n(z)=\frac{1}{\sqrt{1+r^{2n}}}z^n$ (see \eqref{BasisofhARDY}), we have
\[
S_1 e_n = M_z e_n = \sqrt{\frac{1+r^{2n+2}}{1+r^{2n}}} e_{n+1} \qquad (n \in \bZ).
\]
It follows that $S_1^{m}e_n = M_z^{m}e_n$, and a straightforward computation gives
\begin{equation}\label{eqn: S_1men}
S_1^{m}e_n = \sqrt{\frac{1+r^{2(n+m)}}{1+r^{2n}}} e_{n+m} \text{ and } S_1^{*m}e_n = \sqrt{\frac{1+r^{2n}}{1+r^{2(n-m)}}} e_{n-m},
\end{equation}
for all $n \in \bZ$ and $m\geq1$. For each fixed $n$, we have
\begin{align*}
\|S_1^{*m}e_n\|^2=\frac{1+r^{2n}}{1+r^{2(n-m)}} \rightarrow 0 \mbox { as } m \rightarrow \infty.
\end{align*}
Since $S_1$ is a contraction, the operators $S_1^{*m}$ are uniformly bounded. Thus, by density of the linear span of $\{e_n:n\in\mathbb Z\}$, it follows that $S_1^{*m}\to 0$ in the strong operator topology. Similarly, for $m \geq 1$, since $S_2^me_n=(rM_z^{-1})^{m}e_n$, we have
\begin{equation}\label{eqn: S_2men}
S_2^{m}e_n = r^m \sqrt{\frac{1+r^{2(n - m)}}{1+r^{2n}}} e_{n-m} \text{ and } S_2^{*m}e_n = r^m\sqrt{\frac{1+r^{2n}}{1+r^{2(n+m)}}} e_{n+m},
\end{equation}
for all $n \in \bZ$. We conclude that
\begin{align*}
\|S_2^{*m}e_n\|^2=r^{2m}\frac{1+r^{2n}}{1+r^{2(n+m)}} \rightarrow0 \mbox{ as }m \rightarrow\infty,
\end{align*}
for each $n \in \bZ$, and hence, $S_2^{*m}\to 0$ in the strong operator topology.

We are now ready to present a characterization of compact operators on $H^2(\A)$.

\begin{theorem} \label{compathm}
Let $T \in \clb(H^2(\A))$. Then $T$ is compact if and only if
\[
S_i^{*m}TS_j^m \rightarrow 0,
\]
in norm as $m \rightarrow \infty$ for all $i,j=1,2$.
\end{theorem}
\begin{proof}
The necessity follows from a more general fact. Suppose that $T$ is compact. Since $S_1^{*m}\to 0$ and $S_2^{*m}\to 0$ in the strong operator topology as $m \to \infty$, it follows from \cite[1.3 (d), page 3]{BS1990} that
\[
\|S_i^{*m}T\|\longrightarrow0 \qquad (i=1,2),
\]
as $m \rightarrow \infty$. Since $S_j$ is a contraction for $j=1,2$, we have
\begin{align*}
\|S_i^{*m}TS_j^m\|\leq\|S_i^{*m}T\|\,\|S_j^m\|\longrightarrow 0,
\end{align*}
for every $i,j=1,2$ as $m \rightarrow \infty$. This proves the necessity of the condition in the theorem. Conversely, assume that $\|S_i^{*m}TS_j^m\|\longrightarrow0$ as $m \rightarrow \infty$ for $i,j=1,2$. For each $m\geq1$, define
\[
A_m=S_1^mS_1^{*m} \text{ and } B_m=S_2^mS_2^{*m}.
\]
Since $S_1$ and $S_2$ are contractions, $\|A_m\|\leq1$ and $\|B_m\|\leq1$ for every $m\geq 1$. Define
\[
F_m=(I-A_m)(I-B_m) \qquad (m \geq 1).
\]
We claim that $F_m$ is a compact diagonal operator for each $m \geq 1$. Indeed, \eqref{eqn: S_1men} and \eqref{eqn: S_2men} give
\begin{align*}
A_m e_n=\frac{1+r^{2n}}{1+r^{2(n-m)}}e_n\text{ and }  B_m e_n=r^{2m}\frac{1+r^{2n}}{1+r^{2(n+m)}}e_n,
\end{align*}
for every $n \in \bZ$. Therefore,
\[
F_me_n=\left(1-\frac{1+r^{2n}}{1+r^{2(n-m)}}\right)\left(1-r^{2m}\frac{1+r^{2n}}{1+r^{2(n+m)}}\right)e_n,
\]
for each $m\geq1$ and $n \in \mathbb Z$. For each fixed $m \geq 1$, since the scalar coefficient
\[
\left(1-\frac{1+r^{2n}}{1+r^{2(n-m)}}\right)\left(1-r^{2m}\frac{1+r^{2n}}{1+r^{2(n+m)}}\right) \longrightarrow 0,
\]
as $|n|\to\infty$, we conclude that $F_m$ is a compact diagonal operator. Define
\[
K_m = TF_m+F_mT-F_mTF_m \qquad (m\geq 1).
\]
Since $F_m$ is compact, $K_m$ is compact. Moreover,
\[
T-K_m = T-TF_m - F_mT + F_mTF_m = (I-F_m)T(I-F_m).
\]
Since $F_m=(I-A_m)(I-B_m)$, we have
\[
T-K_m = (I-F_m)T(I-F_m)=(A_m+B_m-A_mB_m)T(A_m+B_m-A_mB_m).
\]
Using $A_m=S_1^mS_1^{*m}$ and $B_m=S_2^mS_2^{*m}$, we conclude that
\[
\|T-K_m\| \longrightarrow 0 \text{ as } m \rightarrow \infty,
\]
that is, $T$ is a norm limit of compact operators and is therefore compact. This completes the proof.
\end{proof}

The following result was proved by Abrahamse \cite[Corollary 2.12]{A1974} for multiply connected domains. In the case of the annulus, it follows from Theorems \ref{thm: TO} and \ref{compathm}.

\begin{corollary}
The only compact Toeplitz operator on the annulus is the zero operator.
\end{corollary}
\begin{proof}
Let $T \in \clb(H^2(\A))$ be a compact Toeplitz operator. Then from Theorem \ref{thm: TO} (also see \eqref{eqn MYM = r2Y}), there exist $X,Y \in \clb(H^2(\A))$ such that
\[
M_z^*XM_z=X \text{ and } M_z^*YM_z=r^2Y,
\]
and
\[
X+Y=T.
\]
For each $m \geq 1$, we have
\[
X+r^{2m}Y = M_z^{*m} X M_z^m + M_z^{*m} Y M_z^m = S_1^{*m}TS_1^m,
\]
and hence $\|X+r^{2m}Y\| = \|S_1^{*m}TS_1^m\|$. Since $T$ is compact, Theorem \ref{compathm} implies that $\|X+r^{2m}Y\| \to 0$ as $m \to \infty$. Because $r \in (0,1)$, it follows that $X=0$. Thus $T=Y$. Now for each $m\geq 1$,
\begin{align*}
\|Y\|=\|(rM_z^{-1})^*YrM_z^{-1}\|=\|S_2^*TS_2\|=\|S_2^{*m}TS_2^m\|.
\end{align*}
Again, by Theorem \ref{compathm} and the compactness of $T$, $\|Y\|=\|S_2^{*m}TS_2^m\| \to 0$. Hence $Y=0$, and therefore $T=0$.
\end{proof}

It may be more than a coincidence that the characterization of compact operators on the annulus in Theorem \ref{compathm} resembles the characterization on the polydisc in \cite[Theorem 3.2]{MSS2018}, even though the two domains are quite different. From this perspective, we make the following two comments.

First, we recall the construction of the operator $F_m$ in Theorem \ref{compathm}:
\[
F_m=(I-A_m)(I-B_m) \qquad (m \geq 1).
\]
Similar operators also appear in the proof of the compactness characterization on the polydisc \cite[Theorem 3.2]{MSS2018}. The main difference is that, in the polydisc case, these operators are always finite rank, whereas in the annulus case, the operators $F_m$ are always infinite rank.

Second, we note that assuming only
\[
\|S_i^{*m}TS_i^m\| \longrightarrow 0 \text{ as } m \to \infty \qquad (i=1,2),
\]
in Theorem \ref{compathm} would not yield the desired characterization of compactness. Indeed, consider the unitary operator $J \in \clb(H^2(\A))$ defined by
\[
Je_n=e_{-n}, \quad n \in \mathbb Z.
\]
Clearly, $J$ is not compact. However, it follows from \eqref{eqn: S_1men} and \eqref{eqn: S_2men} that $\|S_1^{*m}JS_1^m\|\leq r^m$ and $\|S_2^{*m}JS_2^m\|\leq r^m$ for each $m\geq 1$. Hence,
\[
\|S_i^{*m} J S_i^m\| \longrightarrow 0 \text{ as } m \to \infty \qquad (i=1,2),
\]
However,
\[
\|S_1^{*m} J S_2^m\|=1,
\]
for every $m\geq1$.

\section{Asymptotic Toeplitz operators}\label{sec: ATO}

Compact perturbations of well-known operators are of broad interest. Here, we study compact perturbations of Toeplitz operators and identify their key properties in the following definition. This class has previously been studied in \cite{BH1982, F1988, MSS2018}.

\begin{definition} \label{asymtopdef}
An operator $T \in \clb(H^2(\A))$ is said to be asymptotic Toeplitz if there exist $J_1, J_2 \in \clb(H^2(\A))$ such that
\[
S_1^{*m}TS_1^m \longrightarrow J_1 \mbox { and } S_2^{*m}TS_2^m \longrightarrow J_2 \mbox{ as } m \rightarrow \infty,
\]
and
\[
S_p^{*m} \bigl(T - J_1 - J_2 \bigr) S_q^m \longrightarrow 0 \mbox{ as } m \rightarrow \infty,
\]
in norm for all $p, q=1,2$.
\end{definition}

We now prove that the asymptotic Toeplitz operators are precisely the compact perturbations of Toeplitz operators.

\begin{theorem}\label{thm: ATO}
An operator $T \in \clb(H^2(\A))$ is asymptotic Toeplitz if and only if $T$ is a compact perturbation of a Toeplitz operator.
\end{theorem}
\begin{proof}
Suppose that $T$ is asymptotic Toeplitz. There exist $J_j \in \clb(H^2(\A))$, $j=1,2$, such that $S_j^{*m} T S_j^m \longrightarrow J_j$ as $m \to \infty$ and
\[
S_p^{*m} \bigl(T- J_1 - J_2 \bigr) S_q^m \longrightarrow 0 \mbox{ as } m \rightarrow \infty,
\]
in norm for all $p, q=1,2$. Since $M_z^{*m}T M_z^m = S_1^{*m}TS_1^m \longrightarrow J_1$ as $m \rightarrow \infty$ in norm, we have
\[
M_z^{*k} J_1 M_z^k=J_1 \qquad (k \geq 1).
\]
In particular, taking $k=1,2$, we have
\[
-M_z^{*2} J_1 M_z^2 + (1+r^2) M_z^* J_1M_z=-J_1 + (1+r^2) J_1 = r^2J_1.
\]
Hence, by Theorem \ref{thm: TO}, $J_1$ is a Toeplitz operator on $H^2(\A)$. Therefore, there exists $\varphi_1 \in L^\infty(\partial \A)$ such that $J_1 = T_{\varphi_1}$. Again, by Theorem \ref{thm: TO}, there exist $X_1, Y_1 \in \clb(H^2(\A))$ such that $J_1=X_1+Y_1$, $M_z^* X_1 M_z = X_1$, and $M_z^* Y_1 M_z = r^2 Y_1$ (see also \eqref{eqn MYM = r2Y}). Combined with the fact that $M_z^* J_1 M_z=J_1$, this implies that
\[
Y_1=0.
\]
It follows from the proof of Theorem \ref{thm: TO}, specifically $(3) \implies (1)$ direction, that $\varphi_1=0$ a.e. on $r\bT$. Indeed, in the $(3) \implies (1)$ direction, $Y_1 = 0$ implies $A_2 = 0$, and hence the corresponding symbol $\widetilde \varphi =0$ in $L^\infty(\bT)$. Thus, $\widetilde{\varphi}_r =0$ on $r\bT$. In the present setting, this implies that $\varphi_1=0$ a.e. on $r \bT$ (here, we also use the fact that a Toeplitz operator on $H^2(\bD)$ is zero if and only if its symbol in $L^\infty(\bT)$ is zero almost everywhere). As in the preceding case, the norm convergence $S_2^{*m} T S_2^m\longrightarrow J_2$ as $m \rightarrow \infty$ implies that $J_2$ is also a Toeplitz operator. Hence, there exist $\varphi_2 \in L^\infty(\partial \A)$ such that $J_2 = T_{\varphi_2}$. As before, we conclude that $\varphi_2 = 0$ a.e. on $ \bT$. Set
\[
\varphi:=\varphi_1 + \varphi_2.
\]
Since $T$ is asymptotic  Toeplitz,
\[
T - J_1 - J_2 = T - T_{\varphi_1} - T_{\varphi_2} = T - T_\varphi,
\]
is compact by Theorem \ref{compathm}. Hence, $T$ is a compact perturbation of a Toeplitz operator.

\noindent Conversely, suppose $T=T_\varphi + K$ for some $\varphi \in L^\infty(\partial \A)$ and a compact operator $K \in \clb(H^2(\A))$. Define $\varphi_1, \varphi_2 \in L^\infty(\partial \A)$ by
\[
\varphi_1 = \varphi \chi_{_\bT} \text{ and } \varphi_2 = \varphi\chi_{_{r\bT}},
\]
and define
\begin{align*}
J_1: = T_{\varphi_1} \text{ and } J_2:= T_{\varphi_2}.
\end{align*}
For $f,g\in H^2(\A)$, since
\begin{align*}
\left\langle M_z^{*m} T_\varphi M_z^m f, g \right\rangle & = \left\langle T_\varphi M_z^mf,M_z^mg\right\rangle
\\
& = \int_{\partial\A}\varphi(z)|z|^{2m}f(z)\overline{g(z)} d\sigma(z)
\\
& = \int_{\bT}\varphi_1(z) f(z)\overline{g(z)} d\sigma(z) + r^{2m}\int_{r \bT} \varphi_2(z) f(z) \overline{g(z)} d\sigma(z),
\end{align*}
we conclude that
\[
M_z^{*m}T_\varphi M_z^m = J_1 + r^{2m} J_2 \qquad (m \geq 1).
\]
Since $0<r<1$, it follows that
\begin{align*}
M_z^{*m}T_\varphi M_z^m \longrightarrow J_1 \mbox{ as } m \rightarrow \infty,
\end{align*}
in operator norm. Again, for $f,g\in H^2(\A)$, since
\begin{align*}
\left\langle S_2^{*m}T_\varphi S_2^mf, g\right\rangle & = \int_{\partial\A}\varphi(z)r^{2m}|z|^{-2m}f(z)\overline{g(z)}\,d\sigma(z)
\\
& = r^{2m} \int_{\mathbb T}\varphi_1(z)f(z)\overline{g(z)}\,d\sigma(z) + \int_{r \bT}\varphi_2(z)f(z)\overline{g(z)}\,d\sigma(z),
\end{align*}
we have
\[
S_2^{*m}T_\varphi S_2^m=r^{2m}J_1 + J_2 \qquad (m \geq 1),
\]
and hence
\begin{align*}
S_2^{*m}T_\varphi S_2^m
\longrightarrow J_2 \mbox { as } m \rightarrow \infty,
\end{align*}
in operator norm. The remaining part follows from Theorem \ref{compathm}.
\end{proof}

\section{Zero-product problem}\label{sec: zero}

We recall an important property of Toeplitz operators on $H^2(\bD)$ established by Brown and Halmos in \cite[Corollary 1]{BH1964}. For symbols $\varphi, \psi \in L^{\infty}(\mathbb T)$,
\[
T_{\varphi}^{H^2(\mathbb D)}T_{\psi}^{H^2(\mathbb D)}=0,
\]
if and only if
\[
\varphi = 0 \text { or } \psi=0 \quad \text{ a.e. on }\mathbb T.
\]
We refer to this as the zero-product problem. We now turn to the annulus case. We begin with a property of Toeplitz operators on the annulus that follows as an application of Theorem \ref{thm: TO}. Suppose $T \in \clb(H^2(\A))$ is a Toeplitz operator. Then
\begin{equation}\label{eqn: TO ident}
S_1^{*m} T S_1^m + S_2^{*m} T S_2^m = (1 + r^{2m}) T \qquad (m \in \bZ_+).
\end{equation}
Indeed, by Theorem \ref{thm: TO}, there exists unique $X, Y \in \clb(H^2(\A))$ such that
\[
T = X + Y,
\]
with $S_1^* X S_1 = X$ and $S_2^* Y S_2 = Y$. For each $m \in \bZ_+$, we have
\[
S_1^{*m} X S_1^m = X \text{ and } S_2^{*m} Y S_2^m = Y.
\]
Since $S_1 S_2 = r I_{H^2(\A)}$, these identities also imply that
\[
S_2^{*m} X S_2^m = r^{2m} X \text{ and } S_1^{*m} Y S_1^m = r^{2m} Y.
\]
Therefore,
\begin{align*}
S_1^{*m} (X + Y) S_1^m + S_2^{*m} (X + Y) S_2^m & = X + r^{2m} Y + r^{2m} X + Y
\\
& = (X + Y) + r^{2m}(X + Y)
\\
& = (1 + r^{2m}) T,
\end{align*}
proving the identity in \eqref{eqn: TO ident}.

A word of caution: The operators $S_1$ and $S_2$ are defined on $H^2(\A)$. However, throughout the following discussion, whenever we say that they act on $L^2(\partial\A)$, we mean their natural extensions to $L^2(\partial \A)$. For instance, $S_1$ acting on $L^2(\partial \A)$ refers to the multiplication operator by $z$ on $L^2(\partial \A)$. It is easy to see that both $S_1$ and $S_2$ are contractions on $L^2(\partial \A)$.

In the context of the following proposition, we note that the product of two Toeplitz operators need not itself be a Toeplitz operator.

\begin{proposition} \label{prod-annulus}
Let $\varphi,\psi,\xi\in L^\infty(\partial \A)$. If $T_\psi T_\varphi=T_\xi$, then
\[
\psi\varphi=\xi \text{ a.e. on } \partial \A.
\]
\end{proposition}
\begin{proof}
It is enough to prove that $\eta = 0$  a.e. on $\partial \A$, where
\[
\eta:= \xi-\psi\varphi.
\]
To prove the result, it suffices to show that $\eta=0$ a.e. on $\partial \mathbb A$. Since $T_\psi T_\varphi=T_\xi$, we have
\begin{align}\label{product_identity}
T_\eta=T_\psi T_\varphi-T_{\psi\varphi} = - P_{H^2(\partial \A)} M_\psi (I_{L^2(\partial \A)} - P_{H^2(\partial \A)})M_\varphi|_{H^2(\partial \A)},
\end{align}
where $M_\varphi$ denotes the operator of multiplication on $L^2(\partial \A)$ by $\varphi$. We first claim that
\begin{align}\label{projection_convergence}
\|(I_{L^2(\partial \A)} - P_{H^2(\partial \A)})(S_i^m g)\|\longrightarrow 0 \quad (m\to\infty),
\end{align}
for every $g\in L^2(\partial \mathbb A)$ and $i=1,2$. Indeed, for each $k\in\mathbb Z$ and $m \in \bZ_+$, we have $S_1^mz^k=z^{m+k} \in H^2(\partial \A)$, and hence
\[
\|(I-P_{H^2(\partial \A)})(S_1^mz^{k}\chi_{\bT})\|=\|(I-P_{H^2(\partial \A)})(z^{m+k}\chi_{\bT})\| \leq \|z^{m+k}\chi_{r\bT}\| = r^{m+k},
\]
and
\[
\|(I-P_{H^2(\partial \A)})(S_1^mz^{k}\chi_{r\bT})\| = \|(I-P_{H^2(\partial \A)})(z^{m+k}\chi_{r\bT})\| \leq \|z^{m+k}\chi_{r\bT}\| = r^{m+k},
\]
and similarly, $S_2^mz^k = r^mz^{k-m}\in H^2(\mathbb A)$ implies
\[
\|(I-P_{H^2(\partial \A)})(S_2^mz^{k}\chi_{\bT})\| = \|(I-P_{H^2(\partial \A)})(r^mz^{k-m}\chi_{\bT})\| \leq r^m,
\]
and
\[
\|(I-P_{H^2(\partial \A)})(S_2^mz^{k}\chi_{r\mathbb T})\| = \|(I-P_{H^2(\partial \A)})(r^mz^{k-m}\chi_{r\mathbb T})\| \leq r^m,
\]
for all $k \in \mathbb Z$ and $m \in \bZ_+$. Since $\{(I-P_{H^2(\partial \A)})S_i^m\}_{m \in \bZ_+}$, $i=1,2$, is uniformly bounded (recall that $S_1$ and $S_2$ are contractions) and $0<r<1$, we conclude that \eqref{projection_convergence} holds. Here we have also used the fact that the linear span of $\{z^k \chi_{\mathbb T}, z^k\chi_{r\mathbb T} : k \in \mathbb Z\}$ is dense in $L^2(\partial \mathbb A)$. Now, for any $f\in H^2(\partial \mathbb A)$ and $m \in \bZ_+$, by using \eqref{product_identity} and \eqref{projection_convergence}, we have
\begin{align*}
\|S_i^{*m}T_\eta S_i^mf\| & = \|S_i^{*m} P_{H^2(\partial \A)} M_\psi (I_{L^2(\partial \A)} - P_{H^2(\partial \A)})M_\varphi (S_i^mf)\|
\\
& \leq \|\psi\|_\infty
\|(I_{L^2(\partial \A)} - P_{H^2(\partial \A)})(\varphi S_i^mf)\|
\\
& = \|\psi\|_\infty \|(I_{L^2(\partial \A)} - P_{H^2(\partial \A)})(S_i^m\varphi f)\|
\\
& \longrightarrow 0,
\end{align*}
as $m \rightarrow \infty$. Therefore,
\begin{equation}\label{product_sot}
S_i^{*m}T_\eta S_i^m \to 0,
\end{equation}
in strong operator topology for $i=1,2$. Given that $T_\eta$ is a Toeplitz operator, \eqref{eqn: TO ident} implies
\[
(1+r^{2m})T_{\eta} = S_1^{*m}T_\eta S_1^m+S_2^{*m}T_\eta S_2^m,
\]
for all $m \in \bZ_+$. Since $0 < r < 1$, we conclude $T_\eta=0$. This immediately implies \cite[Theorem 2.11]{A1974} that $\eta=0$ a.e. on $\partial \mathbb A$, completing the proof of the proposition.
\end{proof}

In particular, we have the following:

\begin{corollary}
Let $\varphi \in L^{\infty}(\partial \mathbb A)$. Then $T_{\varphi}$ is an isometry if and only if $\varphi \in H^{\infty}(\mathbb A)$ and $|\varphi|=1$ a.e. on $\partial \mathbb A.$
\end{corollary}
\begin{proof}
The sufficient part is trivial. Assume that $T_\varphi$ is an isometry. Then
\[
T_{\overline{\varphi}}T_{\varphi}=T_{\varphi}^*T_{\varphi}=I=T_{1}.
\]
It follows from Proposition \ref{prod-annulus} that
\[
\overline{\varphi}\varphi=1,
\]
a.e. on $\partial \A$, and hence $|\varphi|=1$ a.e. on $\partial \A$. Finally, since
\begin{align*}
\|\varphi\|_{L^2(\partial \mathbb A)} = \sqrt{2} = \|T_\varphi 1\|=\|P_{H^2(\partial \mathbb A)} \varphi \|_{L^2(\partial \mathbb A)},
\end{align*}
and $P_{H^2(\partial \mathbb A)}$ is an orthogonal projection, we conclude that $\varphi \in H^2(\mathbb A)$. Thus $\varphi \in H^\infty(\A)$.
\end{proof}

In particular, the only unitary Toeplitz operators are scalar multiples of the identity with modulus one.

As in the Hardy space over $\bD$ case, one can easily see that $\varphi \psi=0$ for some $\varphi, \psi \in L^\infty(\partial \A)$ does not imply that $T_\varphi T_\psi=0$. Indeed, there exists a measurable set $E \subseteq \partial \A$ with $\sigma(E), \sigma(\partial \A \setminus E) > 0$ (for instance, take $E = \bT$ so that $\partial \A \setminus E = r \bT$) such that
\[
T_\varphi T_\psi \neq 0,
\]
where $\varphi = \chi_E$ and $\psi = \chi_{\partial \A \setminus E}$. However, just as in the case of the Hardy space over the disc, we have the following result, which solves the zero-product problem on the annulus. This result was partially proved in \cite[Theorem 1.3]{DN2025}.

\begin{theorem}\label{Zero Pdct. Thm}
Suppose $\varphi, \psi \in L^\infty(\partial \mathbb A)$ and $T_{\varphi}T_{\psi}=0$. Then $\varphi=0$ or $\psi=0$ a.e. on $\partial \mathbb A.$
\end{theorem}
\begin{proof}
By Proposition \ref{prod-annulus},
\[
\psi \varphi=0 \text{ a.e. on } \partial \A.
\]
Define $E_1=\{z \in \partial \mathbb A : \varphi(z)=0\}$ and $E_2=\{z \in \partial \mathbb A : \psi(z)=0\}$. Then $\sigma(E_1)>0$ or $\sigma(E_2)>0$. Without loss of generality, assume that $\sigma(E_1)>0$. If $\varphi = 0$ a.e. on $\partial \A$, the proof is complete. If not, by \cite[Lemma 2.18]{A1974}, we have
\[
\ker T_{\varphi} = \{0\}.
\]
Since $T_{\varphi}T_{\psi}=0$, it follows that $\operatorname{ran} T_{\psi} = \{0\}$, that is, $T_\psi = 0$. This implies \cite[Theorem 2.11]{A1974} $\psi = 0$ a.e. on $\partial \mathbb A$, and completing the proof.
\end{proof}

As a corollary, we have:

\begin{corollary}
Let $\varphi \in L^\infty(\partial \A)$. Then $T_{\varphi}$ is idempotent if and only if $\varphi=1$ or $\varphi=0$ a.e. on $\partial \mathbb A$.
\end{corollary}
\begin{proof}
Let $T_\varphi^2=T_\varphi$, that is, $T_\varphi T_{1-\varphi}=0$. Theorem \ref{Zero Pdct. Thm} immediately implies the conclusion. The converse part is trivial.
\end{proof}

We conclude this section by noting that our results suggest that the operators
\[
S_1:= M_z \qquad\text{and}\qquad S_2:=rM_z^{-1},
\]
play a crucial role in understanding Toeplitz operators on $H^2(\A)$. Of course, these operators are related: both are defined in terms of the shift operator $M_z$, they are connected to the shift on the Hardy space of the disc through Lemma \ref{embeddings}, and they satisfy
\[
S_1 S_2 = r I_{H^2(\A)}.
\]
Our results highlight the significance of these operators and suggest that they may be useful on the Hardy space of the annulus.

\bigskip

\noindent\textbf{Data availability:} Data sharing is not applicable to this article as no datasets were generated or analysed during the current study.

\vspace{0.1in}
\noindent\textbf{AI Declaration:} The results and proofs in this paper were not generated by AI.

\vspace{0.1in}

\noindent\textbf{Acknowledgement:} The authors are indebted to Prof. E. K. Narayanan for many helpful discussions and suggestions throughout the work. The research of the first author is supported by the NBHM Doctoral Fellowship at the Indian Institute of Science. The third author is supported by the NBHM Postdoctoral Fellowship (No. 0204/37(2)/NBHM/R\&D-II/16542). The research of the fourth author is supported in part by the ANRF, DST, Government of India (File No. ANRF/ARGM/2025/000130/MTR).

{}

\end{document}